\documentclass{amsart}
\usepackage{amsmath,amsthm}
\usepackage{amsfonts,amssymb}
\usepackage{accents}
\usepackage{enumerate}
\usepackage{accents,color}
\usepackage{graphicx}

\numberwithin{equation}{section}

\newcommand{\lvt}{\left|\kern-1.35pt\left|\kern-1.3pt\left|}
\newcommand{\rvt}{\right|\kern-1.3pt\right|\kern-1.35pt\right|}

\newtheorem{thm}{Theorem}[section]
\newtheorem{cor}[thm]{Corollary}
\newtheorem{lem}[thm]{Lemma}

\theoremstyle{remark}
\newtheorem{rem}{Remark}[section]

 \def\la{{\langle}}
 \def\ra{{\rangle}}

 \def\a{{\alpha}}
 \def\b{{\beta}}
 \def\g{{\gamma}}

 \def\l{{\lambda}}

 \def\s{\sigma}
 \def\la{{\langle}}
 \def\ra{{\rangle}}

 \def\CD{{\mathcal D}}

 \def\CS{{\mathcal S}}

 \def\CV{{\mathcal V}}

 \def\NN{{\mathbb N}}

      \def\proj{\operatorname{proj}}

\def\lla{\langle{\kern-2.5pt}\langle}      
\def\rra{\rangle{\kern-2.5pt}\rangle}

\newcommand{\wh}{\widehat}

\def\f{\frac}

\graphicspath{{./}}
\begin{document}

\title{Best polynomial approximation on the triangle}

\author{Han Feng}
\address{Department of Mathematics\\
University of Oregon\\
Eugene, Oregon 97403-1222.}\email{hfeng3@uoregon.edu}

\author{Christian Krattenthaler}
\address{Fakult\"at f\"ur Mathematik\\
Universit\"at Wien\\
Oskar-Morgenstern-Platz 1\\
A--1090 Vienna, Austria. }  \email{christian.krattenthaler@univie.ac.at}

\author{Yuan Xu}
\address{Department of Mathematics\\ University of Oregon\\
    Eugene, Oregon 97403-1222.}\email{yuan@uoregon.edu}
  
\thanks{The second author was supported in part by 
the Austrian Science Foundation FWF, 
grant SFB F50 (Special Research Program
``Algorithmic and Enumerative Combinatorics".
The third author was supported in part by NSF Grant DMS-1510296.}

\date{\today}

\begin{abstract} 
Let $E_n(f)_{\a,\b,\g}$ denote the error of best approximation by polynomials of degree at most $n$ in the space 
$L^2(\varpi_{\a,\b,\g})$ on the triangle $\{(x,y): x, y \ge 0, x+y \le 1\}$, where $\varpi_{\a,\b,\g}(x,y) := x^\a y ^\b (1-x-y)^\g$ 
for $\a,\b,\g > -1$. Our main result gives a sharp estimate of $E_n(f)_{\a,\b,\g}$ in terms of the error of best approximation 
for higher order derivatives of $f$ in appropriate Sobolev spaces. The result also leads to a characterization of 
$E_n(f)_{\a,\b,\g}$ by a weighted $K$-functional. 
\end{abstract}

\maketitle

\section{Introduction}
We study best polynomial approximation on the weighted spaces on a triangle. We fix our triangle as 
$$
  \triangle := \{(x,y): x \ge 0, \, y \ge 0, \, x+y \le 1\}
$$ 
and define the weight function to be the Jacobi weight
$$
  \varpi_{\a,\b,\g}(x,y) := x^\a y ^\b (1-x-y)^\g, \quad \a,\b, \g > -1. 
$$
Let $\Pi_n^2$ denote the space of polynomials of degree at most $n$. For $1\le p < \infty$, the error of best approximation
by polynomials in $L^p(\varpi_{\a,\b,\g})$ is defined by
\begin{equation} \label{eq:best-p}
  E_n(f)_{L^p(\varpi_{\a,\b,\g})} : = \inf_{p \in \Pi_n^2} \|f - p\|_{L^p(\varpi_{\a,\b,\g})} ,
\end{equation}
and we replace $L^p(\varpi_{\a,\b,\g})$ by the space $C(\triangle)$ of continuous functions on $\triangle$ when $p =\infty$. 
For $p =2$ we simplify the notation and write
\begin{equation} \label{eq:best-p=2}
    E_n(f)_{\a,\b,\g}: =  E_n(f)_{L^2(\varpi_{\a,\b,\g})}.
\end{equation}

The characterization of best approximation by polynomials on the triangle and, more generally, on the 
$d$-dimensional simplex, has been studied by several authors; see \cite{DHW, DT, To, X05}. In the unweighted 
case $(\a = \b = \g =0$), the best approximation 
on the simplex is characterized by a modulus of smoothness and an equivalent $K$-functional in \cite{DT} and more 
recently in \cite{To} for all $p$ with $1 \le p \le \infty$, as a special case of a more general theorem in the setting of 
polytopes. In these articles, the modulus of smoothness and the $K$-functional are defined as the 
supremum, over all chords in the simplex, of appropriate 
moduli of smoothness or $K$-functionals of one variable over chords. The weighted case is much more difficult 
to characterize, and the complication can be already seen in the case of one variable 
(see \cite{DT}). Currently the only characterization in 
the weighted case is the one given in \cite{X05}, in which the $K$-functional on the triangle is defined by 
\begin{equation} \label{eq:K-func1}
  K_r(f;t)_{L^p(\varpi_{\a,\b,\g})}:= \inf_{g} \left \{ \|f-g\|_{L^p(\varpi_{\a,\b,\g})} + 
     t^r \|(-\CD_{\a,\b,\g})^{r/2} g\|_{L^p(\varpi_{\a,\b,\g})}\right \},
\end{equation}
where $\CD_{\a,\b,\g}$ is the second order differential operator 
\begin{multline} \label{eq:diff-op}
   \CD_{\a,\b,\g} : = 
  \left[\varpi_{\alpha,\beta,\gamma}(x,y)\right]^{-1}  \big [ \partial_1  \varpi_{\alpha+1,\beta,\gamma+1}(x,y)  \partial_1   
     + \partial_2   \varpi_{\alpha,\beta+1,\gamma+1}(x,y) \partial_2 \\
    +  \partial_3 \varpi_{\alpha+1,\beta+1,\gamma}(x,y) \partial_3 ],
\end{multline}
in which $\partial_1$ and $\partial_2$ are the first partial derivatives and we define
$$
  \partial_3 := \partial_2- \partial_1
$$
throughout this paper. In the unweighted case, a more informative modulus of smoothness and its equivalent 
$K$-functional on the simplex were defined in \cite{BX}, and the $K$-functional can be extended to the 
weighted setting by 
\begin{equation} \label{eq:K-func2}
  K^*_r(f;t)_{L^p(\varpi_{\a,\b,\g})}:= \inf_{g} \Big \{ \|f-g\|_{L^p(\varpi_{\a,\b,\g})} + 
      t^r \sum_{1 \le i \le 3} \| \phi_i^r \partial_i^r g \|_{L^p(\varpi_{\a,\b,\g})} \Big\},
\end{equation}
where the $\phi_i$'s are defined by
\begin{equation} \label{eq:phi-i}
  \phi_1(x,y) = \sqrt{x (1-x-y)}, \quad \phi_2(x,y) = \sqrt{y (1-x-y)}, \quad \phi_3(x,y) = \sqrt{x y}.
\end{equation}
For $r =2$, the two $K$-functionals in \eqref{eq:K-func1} and  \eqref{eq:K-func2} are comparable (cf.\ \cite{DHW}), but
the characterization of the best approximation via $K_r^*(f;t)_{L^p(\varpi_{\a,\b,\g})}$ is still open. 

Our study is motivated by the recent work in \cite{X17}, where simultaneous approximation by polynomials on the 
triangle is studied, and the main result involves the errors of best approximation for various derivatives of functions. 
This raises the question of bounding the error of best approximation for a function by those of its derivatives. 
Our main result (see Theorem~\ref{thm:main}) is the estimate
\begin{equation*} 
  E_n(f)_{\a,\b,\g} \le \frac{c}{ n^r} \big[ E_{n-r}(\partial_1^r f)_{\a+r,\b,\g+r} +E_{n-r}(\partial_2^r f)_{\a,\b+r,\g+r}
      +E_{n-r}(\partial_3^r f)_{\a+r,\b+r,\g} \big],
\end{equation*}
where $r$ is a positive integer and $c$ is a constant independent of $n$ and $f$. 
The indices on the right--hand side may look strange at first sight, but this turns out to be natural for at least two 
reasons. First, let $\CV_n(\varpi_{\a,\b,\g})$ be the space of orthogonal polynomials with respect to $\varpi_{\a,\b,\g}$ 
on the triangle; then we have
\begin{align*}
&\partial_1: \CV_n (\varpi_{\a,\b,\g}) \mapsto \CV_{n-1} (\varpi_{\a+1,\b,\g+1}), \\ 
&\partial_2:  \CV_n (\varpi_{\a,\b,\g}) \mapsto \CV_{n-1}(\varpi_{\a,\b+1,\g+1}), \\ 
&\partial_3:   \CV_n (\varpi_{\a,\b,\g}) \mapsto \CV_n(\varpi_{\a+1,\b+1,\g}),
\end{align*}
which shows that $\partial_i^r f$ on the right--hand side of the estimate are being approximated in the right spaces. 
Secondly, the estimate turns out to be what we need to establish the characterization of best approximation by 
$K^*_r(f;t)_{L^2(\varpi_{\a,\b,\g})}$ for all $r$.  

Our paper is organized as follows. Since we work in the $L^2$ setting, we need to deal with the Fourier orthogonal 
expansions on the triangle. This is developed in the next section. The main results on best polynomial  approximation
are presented and proved in the third section. The proof relies on a closed form formula for a family of determinants,
which is established in the fourth section. 

\section{Fourier orthogonal expansions on the triangle}
Since the polynomial of best approximation that attains $E_n(f)_{\a,\b,\g}$ is the $n$-th partial sum of the Fourier
orthogonal expansion with respect to $\varpi_{\a,\b,\g}$, we need to examine orthogonal structure on the triangle. 
For $\a,\b,\g  > -1$, we define an inner product by
$$
  \la f, g \ra_{\a,\b,\g} :=  c_{\a,\b,\g}\int_{\triangle} f(x,y) g(x,y) \varpi_{\a,\b,\g}(x,y) dxdy,
$$
where $c_{\a,\b,\g}$ is chosen so that $\la 1,1\ra_{\a,\b,\g} =1$; 
more precisely, we have
$$
   c_{\a,\b,\g} = \frac{\Gamma(\a+\b+\g+3)}{\Gamma(\a+1)\Gamma(\b+1)\Gamma(\g+1)}. 
$$
Let $\Pi_n^2$ be the space of polynomials of total degree at most $n$ in two variables and let $\CV_n(\varpi_{\a,\b,\g})$ 
be the subspace of orthogonal polynomials of degree $n$ with respect to this inner product. Then
$\Pi_n^d = \bigoplus_{k=0}^n \CV_k(\varpi_{\a,\b,\g})$. The polynomials in $\CV_n(\varpi_{\a,\b,\g})$ are eigenfunctions 
of the second order differential operator $\CD_{\a,\b,\g}$ defined in \eqref{eq:diff-op}; more precisely, 
we have
\begin{equation}  \label{eq:eigen}
 \CD_{\a,\b,\g} P = \l_n P \qquad \hbox{with}\quad \l_n:= - n(n + \a + \b + \g + 2)
\end{equation}
for all $P \in \CV_n(\varpi_{\a,\b,\g})$. An orthogonal basis of the space $\CV_n^{\a,\b,\g}$ can be given in terms of 
the Jacobi polynomials. Let $P^{(\a,\b)}_n$ be the usual Jacobi polynomial of degree $n$ on $[-1,1]$. We adopt 
the normalization 
\begin{equation*} 
J^{\alpha,\beta}_{n}(t) =  \frac{1 }{(n+\a+\beta+1)_{n}}{P}^{(\alpha,\beta)}_{n}(t),
\end{equation*}
where $(a)_n = a(a+1)\cdots (a+n-1)$ is the Pochhammer symbol. By the derivative formula 
for $P_n^{(\a,\b)}$ (cf.\ \cite[(4.5.5)]{Szego}), we then have  $\frac{d}{dt} J^{\a,\b}_n(2t-1)=J_{n-1}^{\a+1,\b+1} (2t-1)$. 
An orthogonal basis of $\CV_n(\varpi_{\a,\b,\g})$ is given by (cf.\ \cite[Section~2.4]{DX}) 
\begin{equation*}
 J_{k,n}^{\a,\b,\g} (x,y) := (x+y)^k J_k^{\a,\b}\left( \frac{y-x}{x+y}\right) J_{n-k}^{2k+\a+\b+1,\g}(1-2x-2y),
    \quad 0 \le k \le n. 
\end{equation*}
More precisely, we have
\begin{equation*} 
 \la J_{k,n}^{\a,\b,\g},  J_{j,m}^{\a,\b,\g} \ra_{\a,\b,\g} = h_{k,n}^{\a,\b,\g} \delta_{j,k}\delta_{n,m}, 
\end{equation*}
where 
\begin{multline}  \label{eq:hkn}
 h_{k,n}^{\a,\b,\g} =  \frac{(\alpha+1)_k\, (\beta+1)_k\, (\a+\b+1)_k\, (\g+1)_{n-k}} 
   {k!\,(n-k)!\, (\a+\b+1)_{2k}\,(\a+\b +2)_{2k} } \\
     \times  \frac{(\a+\b+2)_{n+k}\, (\a+\b+\g+2)_{n+k}}{(\a+\b+\g+2)_{2n}\, (\a+\b+\g+3)_{2n}}. 
\end{multline}
The derivatives of $J_{k,n}^{\a,\b,\g}$ satisfy the following relations (cf.\ \cite[(4.17)]{X17}):   
\begin{align}\label{eq:diffJ}
\begin{split}
  \partial_1  J_{k,n}^{\a,\b,\g}(x,y)  &= -  a_{k,n}^{\a,\b}  J_{k-1,n-1}^{\a+1,\b,\g+1}(x,y) 
    -   J_{k,n-1}^{\a+1,\b,\g+1}(x,y),  \\
  \partial_2  J_{k,n}^{\a,\b,\g}(x,y)  &= a_{k,n}^{\b,\a}  J_{k-1,n-1}^{\a,\b+1,\g+1}(x,y)  - 
          J_{k,n-1}^{\a,\b+1,\g+1}(x,y), \\
     \partial_3 J_{k,n}^{\a,\b,\g}(x,y)  &= J_{k-1,n-1}^{\a+1,\b+1,\g}(x,y)
\end{split}
\end{align}
for $0 \le k \le n$, where 
\begin{equation} \label{eq:akn-bkn}
  a_{k,n}^{\a,\b} :=  \frac{(k+\b)(n+k+\a+\b+1)}{(2k+\a+\b)(2k+\a+\b+1)}.
\end{equation}
The space $\CV_n(\varpi_{\a,\b,\g})$ has several other bases that can be explicitly given. For example,
simultaneous permutation of $\a,\b,\g$ and of $x,y,1-x-y$ in $J_{k,n}^{\a,\b,\g}$ leads to two different orthogonal 
bases of $\CV_n(\varpi_{\a,\b,\g})$. 

The Fourier orthogonal expansion of $f\in L^2(\varpi_{\a,\b,\g})$ is given by
$$
 f =\sum_{m=0}^\infty \sum_{k=0}^{m} \wh{f}_{k,m}^{\a,\b,\g} J^{\a,\b,\g}_{k,m}, \quad \hbox{where} \quad 
      \wh f_{k,m}^{\a,\b,\g}:= \frac{\la f,J^{\a,\b,\g}_{k,m}\ra_{\a,\b,\g}}{h^{\a,\b,\g}_{k,m}}. 
$$
The projection operator $\proj_n^{\a,\b,\g}: L^2(\varpi_{\a,\b,\g}) \mapsto \CV_n(\varpi_{\a,\b,\g})$ and the $n$-th 
partial sum operator $S_n^{\a,\b,\g}: L^2(\varpi_{\a,\b,\g}) \mapsto \Pi_n^2$
are defined by 
$$
 \proj_n^{\a,\b,\g}f:= \sum_{k=0}^{m} \wh{f}_{k,m}^{\a,\b,\g} J^{\a,\b,\g}_{k,m} \qquad \hbox{and}\qquad
   S_n^{\a,\b,\g}f:=\sum_{m=0}^n \proj_m^{\a,\b,\g}f. 
$$
Standard Hilbert space theory shows that the $n$-th partial sum $S_n^{\a,\b,\g} f$ is the least square 
polynomial of degree at most $n$ in $L^2(\varpi_{\a,\b,\g})$; that is, 
\begin{equation} \label{eq:BestApp}
E_n(f)_{\a,\b,\g}:=\inf_{p\in \Pi_n^2} \|f-p\|_{\a,\b,\g}=\|f-S^{\a,\b,\g}_n f\|_{\a,\b,\g},
\end{equation}
where, and throughout the rest of this paper, $\|\cdot\|_{\a,\b,\g} = \|\cdot\|_{L^2(\varpi_{\a,\b,\g})}$. 
As the orthogonal projection from $L^2(\varpi_{\a,\b,\g})$ to $\Pi_n^2$, the partial sum operators are independent 
of the choices of orthogonal bases. The derivatives act commutatively on the partial sum operators in the sense 
that (cf.\ \cite[(4.2.7)]{X17})
\begin{align} \label{eq:proj}
\begin{split}
  \partial_1^r \proj_n^{\a,\b,\g}f&= \proj_{n-r}^{\a+r,\b,\g+r}  \partial^r_1 f, \\
   \partial_2^r \proj_n^{\a,\b,\g}f&= \proj_{n-r}^{\a,\b+r,\g+r}  \partial_2^r f,\\
   \partial_3^r \proj_n^{\a,\b,\g}f&= \proj_{n-r}^{\a+r,\b+r,\g}  \partial_3^r f. 
\end{split}
\end{align}
These relations play an important role for our study. It implies, in particular, the following lemma. 

\begin{lem}
For $r=1,2,3,\ldots$ and $0 \le k \le n -r$, we have
\begin{align} \label{eq:wh-f}
\begin{split}
  (-1)^r \wh{\partial_1^r f}_{k,n-r}^{\a+r,\b,\g+r} & = \sum_{j=0}^r A_{r,j,k,n}^{\a,\b} \wh f^{\a,\b,\g}_{k+j,n},\\
  (-1)^r \wh{\partial_2^r f}_{k,n-r}^{\a,\b+r,\g+r} &= \sum_{j=0}^r (-1)^j A_{r,j,k,n}^{\b,\a} \wh f^{\a,\b,\g}_{k+j,n}, \\
   \wh{\partial_2^r f}_{k,n-r}^{\a+r,\b+r,\g} & = \wh f^{\a,\b,\g}_{k+ r,n},
\end{split}
\end{align}
where
$$
   A_{r,j,k,n}^{\a,\b}=\binom r j \f{(k+\b+1)_j\,(n+k+\a+\b+2)_j}{(2k+\a+\b+j+1)_j\,(2k+\a+\b+r+2)_j}. 
$$
\end{lem}

\begin{proof}
By the first identity in \eqref{eq:diffJ}, we obtain
\begin{align*}
 \partial_1 \proj_n^{\a,\b,\g} f   & = 
   - \sum_{k=0}^n \wh{f}_{k,n}^{\a,\b,\g} \left(a_{k,n}^{\a,\b}  J_{k-1,n-1}^{\a+1,\b,\g+1}  
    + J_{k,n-1}^{\a+1,\b,\g+1}  \right) \\
 & = - \sum_{k=0}^{n-1} \left( \wh{f}_{k,n}^{\a,\b,\g}+ a_{k+1,n}^{\a,\b} \wh{f}_{k+1,n}^{\a,\b,\g}\right) J_{k,n-1}^{\a+1,\b,\g+1}.
\end{align*}
Hence, by the first identity in \eqref{eq:proj}, we conclude that 
$$
- \, \wh{\partial_1 f}^{\a+1,\b,\g+1}_{k,n-1}= \wh f^{\a,\b,\g}_{k,n} + a^{\a,\b}_{k,n}\wh f^{\a,\b,\g}_{k-1,n},
$$
which is the first identity in \eqref{eq:wh-f} with $r =1$, as $A_{1,0,k,n} = 1$ and $A_{1,1,k,n} = a^{\a,\b}_{k+1,n}$. 
The case $r > 1$ follows by induction. Indeed, assume that \eqref{eq:wh-f} holds up to the $r$th derivative. The above 
consideration with $\a, \g$ replaced by $\a+r, \g+r$, $n$ replaced by $n-r$, and $f$ replaced by $\partial_1^r f$ 
then shows that
\begin{align*}
     (-1)^{r+1} \wh{\partial_1^{r+1} f}^{\a+r+1,\b,\g+r+1}_{k,n-r-1} & = 
    (-1)^r \wh {\partial_1^r f}^{\a+r,\b,\g+r}_{k,n-r} + (-1)^r a^{\a+r,\b}_{k+1,n-r}\wh {\partial_1^r f}^{\a+r,\b,\g+r}_{k+1,n-r} \\
  & = \sum_{j=0}^r A_{r,j, k,n}^{\a,\b} \wh f_{k+j,n}^{\a,\b,\g}  + 
      a^{\a+r,\b}_{k+1,n-r}\sum_{j=0}^r A_{r,j, k+1,n}^{\a,\b} \wh f_{k+j+1,n}^{\a,\b,g},
\end{align*}
where the second identity follows from the induction hypothesis. By reordering the second 
sum, we see that the first identity in \eqref{eq:wh-f} holds for the $(r+1)$th derivative with coefficients given by 
$$
A_{r+1,j,k,n}^{\a,\b}= A_{r,j,k,n}^{\a,\b} + a_{k+1,n+r}^{\a+r,\b}A_{r,j-1,k+1,n}^{\a,\b},
$$
from which the explicit formula for $A_{r, j , k, n}^{\a,\b}$ follows by induction and a straightforward computation. 

The proof of the second identity in \eqref{eq:wh-f} is similar, and the third identity follows immediately 
from the third identity in \eqref{eq:diffJ} and the third identity in \eqref{eq:proj}. 
\end{proof}

For $n \ge k+ 2 r -1$, the first two equations in \eqref{eq:wh-f} lead to the system of linear equations 
\begin{equation}\label{M_k}
M_r(k,n) \left[  \begin{array}{c}
                     \wh f^{\a,\b,\g}_{k,n} \\
                     \wh f^{\a,\b,\g}_{k+1,n} \\
                     \vdots\\
                     \wh f^{\a,\b,\g}_{k+2r-1,n} 
                          \end{array} \right]
          =\left[\begin{array}{c}
                     \wh{\partial^r_1 f}^{\a+r,\b,\g+r}_{k,n-r} \\
                     \vdots \\
                     \wh{\partial^r_1 f}^{\a+r,\b,\g+r}_{k+r-1,n-r} \\
                     \wh{\partial^r_2 f}^{\a,\b+r,\g+r}_{k,n-r} \\
                     \vdots \\
                     \wh{\partial^r_2 f}^{\a,\b+r,\g+r}_{k+r-1,n-r}
                \end{array}\right],
\end{equation}
where $M_r(k,n)$ is the $2r \times 2r$ matrix defined by 
$$
M_r(k,n):=\begin{pmatrix} 
 A_{r,j-i, k+i, n}^{\a,\b} &\text{for }0\le i< r\\
 (-1)^{j - i + r} A_{r, j-i+r, k+i-r,n}^{\b,\a} &\text{for }r\le i< 2 r-1
\end{pmatrix}_{0\le i,j\le 2r -1}.
$$
The matrix $M_r(k,n)$ is invertible. In fact, its determinant has a closed form as seen in the following lemma. 

\begin{lem} \label{lem:2.2}
For $r \in \NN$, $n \ge k \ge 0$ and $\a,\b > -1$, 
\begin{equation}\label{eq:detMr}
   \det M_r(k,n) = (-1)^{r^2} \prod_{j=1}^r  \f{(n+k+\a+\b+j+1)_r}{(2k+\a+\b+2r+j)_r}.
\end{equation}
\end{lem}

\begin{proof}
The proof relies on Theorem~\ref{thm:1} proved in Section~\ref{sec:4}. Here we show how it can be deduced 
from Theorem~\ref{thm:1}. We define
$$
   A_{r,j, k}^{\a,\b} = \binom r j \f{(k+\b+1)_j}{(2k+\a+\b+j+1)_j\,(2k+\a+\b+r+2)_j}, 
$$
which is $ A_{r,j, k,n}^{\a,\b}$ without its factor that depends on $n$, and we define the matrix $M_r(k)$ by 
$$
M_r(k):=\begin{pmatrix} 
 A_{r,j-i, k+i}^{\a,\b} &\text{for }0\le i< r\\
 (-1)^{j - i + r} A_{r, j-i+r, k+i-r}^{\b,\a} &\text{for }r\le i< 2 r-1
\end{pmatrix}_{0\le i,j\le 2r -1}.
$$
Using the fact that 
$$
  (n+k + i+ \a + \b +2)_{j-i} = \frac{(n+k + \a + \b +2)_j}{(n+k + \a + \b +2)_i},
$$
it is not difficult to see that the factors containing $n$ in the matrix $M_r(k,n)$ can be factored out. More precisely, 
define two diagonal matrices by
$$ 
  R_r(k,n) =  \mathrm{diag}\left\{  (n+k + \a + \b +2)_j :  0 \le j \le 2 r-1\right\}
$$
and 
$$
  L_r(k,n) = \left[ \begin{matrix} \Lambda_r  & 0 \\ 0 & \Lambda_r \end{matrix} \right], \qquad 
\Lambda_r = \mathrm{diag}\left\{ \frac{1}{(n+k + \a + \b +2)_i}:  0 \le i \le r-1 \right\}.
$$
Then it is easy to verify that 
$$
   M_r (k) = L_r(k,n) M_r (k,n) R_r(k,n).
$$
Evaluating the determinants of $L_r(k,n)$ and $R_r(k,n)$, we see that \eqref{eq:detMr} reduces to 
$$
  \det M_r (k ) =  \frac{(-1)^{r^2}}{\prod_{j=1}^r(2k+\a+\b+2r+j)_r}.
$$
This last identity is a special case of Theorem~\ref{thm:1}, as can be seen by setting $s_1 = k+\a +1$, 
$s_2 = k + \b +1$, and $r_1 = r_2 = r$ in \eqref{eq:1}.
\end{proof}

Since the matrix $M_r(k,n)$ is invertible, the system of equations \eqref{M_k} can be solved to give an 
expression for $\wh {\partial^r_i f}^{\a+r,\b,\g+r}_{k,n-r}$ as a sum of $\wh f^{\a,\b,\g}_{k+j,n}$ over $j$, which
will be needed in the proof of our main result. 

\begin{lem} \label{lem:2.3}
Let $ r \in \NN$. For $n \ge k + 2r -1$, we have
$$
  \wh f_{k,n}^{\a,\b,\g} = \sum_{\ell=1}^{r} B_{\ell,1} (k,n) \wh{\partial_1^r f}^{\a+r,\b,\g+r}_{k+\ell-1, n-r}+
      \sum_{\ell=1}^{r} B_{\ell,2}(k,n) \wh{\partial_2^r f}^{\a,\b+r,\g+r}_{k+\ell-1,n-r}, 
$$
where the constants $B_{\ell,i}(k,n)$ satisfy 
$$
    \left | B_{\ell,i} (k,n) \right | \le c \left( \frac{n}{k+1} \right)^{\ell -1},  \quad 0 \le \ell \le r-1, \quad i =1,2,
$$ 
where $c$ is a constant independent of $n$ and $k$. 
\end{lem}

\begin{proof}
We only need to solve for the first element, $\wh f_{k,n}^{\a,\b,\g}$, in the linear system \eqref{M_k}.
For $1 \le \ell \le 2r$, let $M_r^{\ell,1}(k,n)$ be the matrix formed by eliminating the first column and
the $\ell$-th row from the matrix $M_r(k,n)$. By Cramer's rule, 
$$
B_{\ell,i}(k,n) = \frac{\det M_r^{\ell,1}(k,n)} {\det M_r(k,n)},
$$
where $1 \le \ell \le r$ for $i=1$ and $r+1 \le \ell \le 2 r$ for $i =2$.   
From the explicit expression for $M_r(k,n)$, it follows readily that 
\begin{equation} \label{eq:det-bound}
\left | \det M_r(k,n) \right | \sim \left ( \frac n {k+1} \right)^{r^2}, \qquad 0 \le k \le n,
\end{equation}
where $A \sim B$ means that there exist positive constant $c_1$ and $c_2$ such that $c_1 \le A/B \le c_2$. We now estimate
$|\det M_r^{\ell,1}(k,n)|$ from above. 

We first assume $1 \le \ell \le r$. Let $A_{i,j}$ denote the $(i,j)$-entry of a matrix $A$. The entries of $M_r(k,n)$ and
$M_r^{\ell,1}(k,n)$ are indexed by $i, j =1,\ldots, 2r$ and $2r-1$, respectively. Then 
$$
  M_r^{\ell,1}(k,n)_{i,j} =  \begin{cases} 
                             M_r(k,n)_{i, j+1}, & \hbox{$1 \leq i\leq  \ell-1$;} \\
                             M_r(k,n)_{i+1, j+1}, & \hbox{$\ell \leq i\leq 2r-1 $.}   \end{cases}
$$ 
From the explicit formula for $A^{\a,\b}_{r,j-i,k+i,n}$, it follows that 
$$
  M_r(k,n)_{i,j} = A^{\a,\b}_{r,j-i,k+i-1,n}\lesssim \left (\f n {k+1} \right)^{j-i}
$$
for $1\leq i\leq r$ and $i\leq j\leq i+r$, where $\lesssim$ means that the inequality holds up to a constant
independent of $k$ and $n$, and  
$$
\left | M_{r}(k,n)_{i,j} \right |= A^{\b,\a}_{r,j-i+r,k+i-r-1,n}\lesssim \left(\f n {k+1} \right)^{j-i+r}
$$
for $r+1 \leq i\leq 2r$ and $i-r\leq j\leq i$. Consequently,  we deduce that, for $i,j= 1, 2, \ldots, 2r-1$,
we have
\begin{equation} \label{eq:coeff-bound}
 \left | M_{r}^{\ell,1}(k,n)_{i,j} \right |\lesssim   \begin{cases}
                          (\f n {k+1})^{j+1-i}, & \hbox{$ 1 \leq i \le \ell-1$,}\\
                          (\f n {k+1})^{j-i}, & \hbox{$\ell \leq i\leq r-1$,} \\ 
                          (\f n {k+1})^{j-i+r}, & \hbox{$r \leq i\leq 2r-1$}. \end{cases}
\end{equation}

Now, by the definition of the determinant, 
\begin{equation*}
   \det M_r^{1,\ell}(k,n)=\sum_{ \s \in \CS_{2r-1}}\mathrm{sign}(\s) \prod_{i=1}^{2r-1}M^{1,\ell}_r(k,n)_{i,\s(i)},
\end{equation*}
where $\CS_{2r-1}$ is the set of all permutations of $\{1, 2,\ldots, 2r-1\}$. For each $\s\in S_{2r-1}$, we then
obtain, using \eqref{eq:coeff-bound},
\begin{multline*}
\left| \prod_{i=1}^{2r-1}M_r^{(1,\ell)}(k,n)_{i, \s(i)} \right |   \lesssim    
   \prod_{i=1}^{\ell-1}   \left(\f n {k+1} \right)^{\s(i)+1-i}
\times\prod_{i=\ell}^{r-1}  \left(\f n {k+1} \right)^{\s(i)-i} \\
      \qquad 
\times\prod_{i=r}^{2r-1}  \left(\f n {k+1} \right)^{\s(i)-i+r}  
        =  \left(\f n {k+1} \right)^{r^2+\ell-1},
\end{multline*}
where the last equation follows from $\sum_{i=1}^{2r-1}(\s(i)-i)=0$. Consequently, we conclude that 
$$
 \left| \det M_r^{1,\ell}(k,n)  \right | \lesssim  \left( \f n {k+1} \right )^{r^2+\ell-1}.
$$
Together with \eqref{eq:det-bound}, this establishes the desired estimate for $B_{\ell,1}(k,n)$. 

The estimate for $B_{\ell,2}(k,n)$ can be proved in the same way. Indeed, if $\ell$ satisfies $r +1 \le \ell \le 2r$, we 
may exchange the rows of $M_r^{\ell,1}(k,n)$ so that the last $r-1$ rows become the first $r-1$ rows, which does 
not change the value of the absolute value of the determinant. Furthermore, since our proof relies only on absolute
values of the entries, the signs $(-1)^{j-i+r}$ in the entries of $M_r(k,n)$ can be ignored. 
\end{proof}

\section{Best polynomial approximation in weighted space}
For $\a,\b,\g > -1$, let $W_2^r(\varpi_{\a,\b,\g})$ denote the Sobolev space defined by
$$
  W_2^r(\varpi_{\a,\b,\g}) = \{f \in L^2(\varpi_{\a,\b,\g}): \phi_i^r \partial_i^r f \in L^2(\varpi_{\a,\b,\g}), \, i =1, 2, 3\},
$$
where the $\phi_i$'s are defined in \eqref{eq:phi-i}. Since $\phi_i^r g \in L^2(\varpi_{\a,\b,\g})$ is equivalent with 
the assertion that $g \in 
L^2(\varpi_{\a+r,\b,\g+r})$ for $i =1$, $g \in L^2(\varpi_{\a,\b+r,\g+r})$ for $i =2$, and $g \in L^2(\varpi_{\a+r,\b+r,\g})$ 
for $i =3$, it follows from \eqref{eq:diffJ} that $J_{k,n}^{\a,\b,\g} \in W_2^r(\varpi_{\a,\b,\g})$. Our main result is the 
following theorem. 

\begin{thm} \label{thm:main}
Let $\a,\b,\g> -1$, and let $r$ be a positive integer. For $f \in W_2^r(\varpi_{\a,\b,\g})$, 
we have
\begin{equation*} 
  E_n(f)_{\a,\b,\g} \le \frac{c}{ n^r} \big[ E_{n-r}(\partial_1^r f)_{\a+r,\b,\g+r} +E_{n-r}(\partial_2^r f)_{\a,\b+r,\g+r}
      +E_{n-r}(\partial_3^r f)_{\a+r,\b+r,\g} \big]
\end{equation*}
for $n \ge 3 r$, where $c$ is a constant independent of $n$ and $f$. 
\end{thm}

\begin{proof}
By Parseval's identity, 
$$
  \left[ E_n(f)_{\a,\b,\g}\right]^2 = \|f- S_n^{\a,\b,\g}\|_{\a,\b,\g}^2 
     = \sum_{m=n+1}^\infty \sum_{k=0}^m |\wh f_{k,m}^{\a,\b,\g}|^2  h_{k,m}^{\a,\b,\g}.
$$
In order to bound this series, we consider two cases. First of all, for $m/3 \le k \le m$, the third identity in \eqref{eq:wh-f} shows
that, for $m \ge r$, we have
\begin{equation} \label{eq:i=3}
   |\wh f_{k,m}^{\a,\b,\g}|^2 h_{k,m}^{\a,\b,\g} = |\wh {\partial_3 f}_{k-r,m-r}^{\a+r,\b+r,\g}|^2 h_{k,m}^{\a,\b,\g} 
      \sim m^{- 2r} |\wh {\partial_3 f}_{k-r,m-r}^{\a+r,\b+r,\g}|^2 h_{k-r,m-r}^{\a+r,\b+r,\g},
\end{equation}
since, by the explicit formula for $h_{k,m}^{\a,\b,\g}$, it is is easy to see that
$$
    h_{k-r,m-r}^{\a+r,\b,\g+r} = (k+\a+\b+1)_r\, (k-r)_r\,  h_{k,m}^{\a,\b,\g}  \sim k^{2r}  h_{k,m}^{\a,\b,\g}.
$$
Consequently, it follows that 
\begin{align*}
\sum_{m=n+1}^\infty \sum_{k=\lfloor \frac m 3 \rfloor}^m |\wh f_{k,m}^{\a,\b,\g}|^2 h_{k,m}^{\a,\b,\g} 
 & \lesssim  \sum_{m=n+1}^\infty m^{- 2r} \sum_{k=\lfloor \frac m 3 \rfloor}^m |\wh {\partial_3^r f}_{k-r,m-r}^{\a+r,\b+r,\g}|^2
    h_{k-r,m-r}^{\a+r,\b+r,\g} \\
  & \lesssim n^{- 2r} \sum_{m=n-r+1}^\infty \sum_{k=0}^m |\wh {\partial_3^r f}_{k,m}^{\a+r,\b+r,\g}|^2 h_{k,m}^{\a+r,\b+r,\g}\\
  & \lesssim n^{- 2r} \left [ E_{n-r} (\partial_3^r f)_{\a+r,\b+r,\g} \right ]^2, 
\end{align*}
where the last step follows again by Parseval's identity. For the case $0 \le k \le m/3$, we use the 
elementary estimates
\begin{equation} \label{eq:i=1}
 h_{k,m}^{\a,\b,\g} \sim \left ( \frac{k+1}{m} \right)^2 h_{k+1,m}^{\a,\b,\g} \quad\hbox{and}\quad  
   h_{k,m}^{\a,\b,\g} \sim \frac{1}{(m+1-k) m} h_{k,m-1}^{\a+1,\b,\g+1},
\end{equation}
derived from the explicit formula 
for $h_{k,m}^{\a,\b,\g}$, and use them iteratively to obtain
$$
    h_{k,m}^{\a,\b,\g}  \sim \left( \frac{k+1}{m} \right)^{2 \ell -2} \frac{1}{ (m+1-k)^{r} m^{r}} h_{k+\ell-1,m-r}^{\a+r,\b,\g+r}\le  
   \left( \frac{k+1}{m} \right)^{2\ell-2} \frac{1}{ m^{2r}} h_{k+\ell-1,m-r}^{\a+r,\b,\g+r}. 
$$
The same estimate holds if the last term is $h_{k+j,m-r}^{\a,\b+r,\g+r}$. We now use Lemma~\ref{lem:2.3}, whose 
assumption is satisfied since $m \ge n+1 \ge 3 r+1$ implies $m \ge m/3+2r-1$, to obtain 
$$
  \left| \wh f_{k,m}^{\a,\b,\g}  \right |^2 \lesssim  \left ( \frac{m}{k+1} \right)^{2\ell-2} 
        \sum_{\ell=1}^{r} \left( \left| \wh{\partial_1^r f}^{\a+r,\b,\g+r}_{k+\ell-1, n-r} \right |^2 + 
            \left| \wh{\partial_2^r f}^{\a,\b+r,\g+r}_{k+\ell-1, n-r} \right |^2 \right).
$$
Putting these estimates together, 
that 
$$
  \left| \wh f_{k,m}^{\a,\b,\g}  \right |^2 \lesssim  \left ( \frac{m}{k+1} \right)^{2\ell-2} 
        \sum_{\ell=1}^{r} \left( \left| \wh{\partial_1^r f}^{\a+r,\b,\g+r}_{k+\ell-1, n-r} \right |^2 + 
            \left| \wh{\partial_2^r f}^{\a,\b+r,\g+r}_{k+\ell-1, n-r} \right |^2 \right).
$$
Combining these estimates, we see that 
\begin{multline} \label{eq:whf-whDf}
\left| \wh f_{k,m}^{\a,\b,\g}  \right |^2 h_{k,m}^{\a,\b,\g} \lesssim \frac{1}{m^{2r}}  
     \sum_{\ell=1}^{r}  \left[\Big | \wh{\partial_1^r f}^{\a+r,\b,\g+r}_{k+\ell-1, m-r} \Big |^2 h^{\a+r,\b,\g+r}_{k+\ell-1, m-r} \right. \\
      \left.  +  \Big | \wh{\partial_2^r f}^{\a,\b+r,\g+r}_{k+\ell-1, m-r} \Big |^2 h^{\a+r,\b+r,\g}_{k+\ell-1, m-r}\right], 
\end{multline}
from which we deduce immediately that
\begin{align*}
  \sum_{m=n+1}^\infty \sum_{k= 0}^{\lfloor \frac m 3 \rfloor} |\wh f_{k,m}^{\a,\b,\g}|^2 h_{k,m}^{\a,\b,\g} 
       & \lesssim  \frac{1}{n^{2r}}  \left ( \sum_{m=n-r+1}^\infty 
          \sum_{k= 0}^m \Big | \wh{\partial_1^r f}^{\a+r,\b,\g+r}_{k, m} \Big |^2 h^{\a+r,\b,\g+r}_{k,m} \right.\\ 
         & \qquad\qquad \left. +  \sum_{m=n-r+1}^\infty \sum_{k= 0}^m 
            \Big | \wh{\partial_2^r f}^{\a,\b+r,\g+r}_{k,m} \Big |^2 h^{\a+r,\b+r,\g}_{k,m} \right)\\
         & \lesssim  n^{- 2r} \left( \left[ E_{n-r} (\partial_1^r f)_{\a+r,\b,\g+r}\right ]^2 + \left[ E_{n-r} (\partial_2^r f)_{\a,\b+r,\g+r}\right ]^2 \right).
\end{align*}
The proof is finally completed by putting the estimates in the two cases together. 
\end{proof}

\begin{cor}
Let $\a,\b,\g> -1$, and let $r$ be a positive integer. For $f \in W_2^r(\varpi_{\a,\b,\g})$, 
we have
\begin{equation*} 
  E_n(f)_{\a,\b,\g} \le \frac{c}{ n^r} \sum_{i=1}^3 \| \phi_i^r \partial_i^r f \|_{L^2(\varpi_{\a,\b,\g})}. 
\end{equation*}
\end{cor}

\begin{proof}
By its definition, $E_n(f)_{\a,\b,\g} \le \| f\|_{L^2(\varpi_{\a,\b,\g})}$ if the approximating polynomial is 
chosen to be zero. Application of this estimate on the right--hand side of the estimate in 
Theorem~\ref{thm:main} yields the above estimate, since $\partial_1^r f \in L^2(\varpi_{\a+r,\b,\g+r})$ is equivalent 
with $\phi_1^r \partial_1^r  \in L^2(\varpi_{\a,\b,\g})$,
and a similar equivalence works for $\partial_2^r f$ and $\partial_3^r f$.  
\end{proof}

As an immediate corollary of the above estimate, we also obtain a characterization of the best approximation
by polynomials by the $K$-functional $K_r^*(f;t)_{L^2(\varpi_{\a,\b,\g})}$, defined in \eqref{eq:K-func2}. 

\begin{thm} 
Let $\a, \b, \g > -1$. For $r \in \NN$ and $f \in L^2(\varpi_{\a,\b,\g})$,  
we have
\begin{equation} \label{eq:Jackson}
  E_n(f)_{\a,\b,\g} \le c \, K_r^*(f; n^{-1})_{L^2(\varpi_{\a,\b,\g})}, 
\end{equation}
and, conversely,  
\begin{equation} \label{eq:Bernstein}
    K_r^*(f; n^{-1})_{L^2(\varpi_{\a,\b,\g})} \le   c  n^{-r} \sum_{k=0}^n (k+1)^{r-1} E_k(f)_{\a,\b,\g}.
\end{equation}
\end{thm}

\begin{proof}
For simplicity, let $\|\cdot\| = \|\cdot\|_{L^2(\varpi_{\a,\b,\g})}$ in this proof. The partial sum $S_n^{\a,\b,\g}f $ 
defines a linear operator that satisfies $\|S_n^{\a,\b,\g} f\| \le \|f\|$ by Parseval's identity. Hence, for 
$g \in W_2^r(\varpi_{\a,\b,\g})$, the triangle inequality gives 
\begin{align*}
  \|f - S_n^{\a,\b,\g} f\| & \, \le \|f -g\| + \|S_n^{\a,\b,\g} f -S_n^{\a,\b,\g} g\|  + \|S_n^{\a,\b,\g} g - g \| \\
   &\, \le 2 \|f-g\| +  \frac{c}{ n^r} \sum_{i=1}^3 \| \phi_i^r \partial_i^r g \| \
     \le c_1 \left ( \|f-g\|+ \frac{1}{ n^r}\sum_{i=1}^3 \| \phi_i^r \partial_i^r g \| \right ),
\end{align*}
where $c_1 = \max\{2, c\}$ is independent of $g$ and $n$. The direct estimate \eqref{eq:Jackson} follows by
taking the infimum over $g$.

The inverse estimate \eqref{eq:Bernstein} can be derived from the first inequality in \eqref{eq:K-equiv} below
and the inverse estimate of the $K$-functional $K_r(f; t)_{L^2(\varpi_{\a,\b,\g})}$ established in 
\cite[Theorem~5.3]{X05}.
\end{proof}

Recall that the $K$-functional $K_r(f;t)_{L^2(\varpi_{\a,\b,\g})}$ in \eqref{eq:K-func1} is defined in 
terms of the operator $(-\CD_{\a,\b,\g})^{r/2}$, where $\CD_{\a,\b,\g}$ is the second order differential 
operator \eqref{eq:diff-op}. For an integer $r$, the operator $(- \CD_{\a,\b,\g})^{r/2}$ is defined via its 
orthogonal expansion (see \eqref{eq:eigen})
$$
  ( - \CD_{\a,\b,\g})^{r/2} g = \sum_{m=1}^\infty 
     \l_{m}^r \sum_{k=0}^m \wh g_{k,m}^{\a,\b,\g} J_{k,m}^{\a,\b,\g}, 
$$
where $\l_m = -m (m+\alpha+\beta+1)$. Notice that the sum starts with $m =1$ since $\l_0 = 0$. 
The $K$-functional $K_r(f;t)_{L^p(\varpi_{\a,\b,\g})}$ characterizes the best approximation by polynomials for 
all $p \ge 1$, and, more generally, on the $d$-dimensional simplex \cite{X05}. These two 
$K$-functionals can be compared as follows. 

\begin{thm}
Let $r \in \NN$. For $f\in L^2(\varpi_{\a,\b,\g})$, we have
\begin{equation} \label{eq:K-equiv}
 c_1 K_r^*(f;t)_{L^2(\varpi_{\a,\b,\g})} \le K_r (f;t) \le c_2 \left(K_r^*(f;t)_{L^2(\varpi_{\a,\b,\g})} +  t^r \|f\|_{L^2(\varpi_{\a,\b,\g})} \right),
\end{equation}
where $c_1$ and $c_2$ are positive constants independent of $f$. 
\end{thm}

\begin{proof}
We first prove the inequality on the left, by establishing the inequality 
\begin{equation} \label{eq:der-bound}
  \left \| \phi_i^r \partial_i^r g \right \|_{\a,\b,\g} \le c \left \|(-\CD_{\a,\b,\g})^{r/2} g \right\|_{\a,\b,\g}, \qquad i =1, 2, 3. 
\end{equation}
Once \eqref{eq:K-equiv} is established, we see that we could restrict ourselves to  $g \in L^2(\varpi_{\a,\b,\g})$.
By Parseval's identity, we have
$$
  \left \|(-\CD_{\a,\b,\g})^{r/2} g \right\|_{\a,\b,\g}^2 = \sum_{m=1}^\infty  |\lambda_m|^r 
     \sum_{k=0}^m \left|\wh g_{k,m}^{\a,\b,\g}\right|^2 h_{k,m}^{\a,\b,\g}.
$$
For $i=3$, we need to examine the proof of the relations in \eqref{eq:i=3}, which implies that 
\begin{align*}
  \big \| \phi_3^r \partial_3^r g \big \|_{\a,\b,\g}^2 \, & = \sum_{m=0}^\infty \sum_{k=0}^m
       \Big |\wh {\partial_3 g}_{k,m}^{\a+r,\b+r,\g}\Big |^2 h_{k,m}^{\a+r,\b+r,\g} \\
    & \sim   \sum_{m=0}^\infty  \sum_{k=0}^m k^{2m} \Big |\wh g_{k+r,m+r}^{\a,\b,\g}\Big|^2 h_{k+r,m+r}^{\a,\b,\g} \notag\\
    &  \lesssim  \sum_{m=r}^\infty  m^{2r} \sum_{k=0}^m \Big |\wh g_{k,m}^{\a,\b,\g}\Big|^2 h_{k,m}^{\a,\b,\g} 
         \lesssim  \left \|(-\CD_{\a,\b,\g})^{r/2} g \right\|_{\a,\b,\g}^2, \notag
\end{align*}
since $|\l_m| \sim m^2$. This proves \eqref{eq:der-bound} for $i=3$. For $i=1,2$, we need the estimates
$$
 \left |A_{r,j,k,m}^{\a,\b} \right | \sim \left(\frac{m}{k+1}\right)^j, \qquad h_{k,m-r}^{\a+r,\b,\g+r} \sim h_{k,m-r}^{\a,\b+r,\g+r}\sim
        \left(\frac{k+1}{m}\right)^{2j} m^{2r} h_{k+j,m}^{\a,\b,\g}; 
$$
while the first one is immediate, the second follows from iterations of relations in  \eqref{eq:i=1}. These two 
estimates imply, if one also uses the first identity in \eqref{eq:wh-f}, that 
\begin{align*}  
     \| \phi_1^r \partial_1^r g \|_{\a,\b,\g}^2 & =  
      \sum_{m=0}^\infty \sum_{k=0}^m \Big |\wh {\partial_1^r g}_{k,m}^{\a+r,\b,\g+r} \Big |^2 h_{k,m}^{\a+r,\b,\g+r} \\
      & =  \sum_{m=0}^\infty \sum_{k=0}^m \left | \sum_{j=0}^r A_{r,j,k,m+r}^{\a,\b} \wh g^{\a,\b,\g}_{k+j,m+r} \right |^2
            h_{k,m}^{\a+r,\b,\g+r} \\
      & \lesssim  \sum_{m=r}^\infty m^{2r} \sum_{k=0}^m  \Big| \wh g^{\a,\b,\g}_{k,m} \Big|^2 h_{k,m}^{\a,\b,\g}    
       \lesssim \left \|(-\CD_{\a,\b,\g})^{r/2} g \right\|_{\a,\b,\g}^2.
\end{align*}
This establishes \eqref{eq:der-bound} for $i =1$. The proof for the the case $i=2$ is similar. 

We now prove the inequality on the right in \eqref{eq:K-equiv}. First we prove the inequality 
\begin{equation} \label{eq:der-bound2}
 \left \|(-\CD_{\a,\b,\g})^{r/2} g \right\|_{\a,\b,\g} \le c \left(\sum_{i=1}^3 \left \| \phi_i^r \partial_i^r g \right \|_{\a,\b,\g}
     + \|g\|_{\a,\b,\g} \right). 
\end{equation} 
The proof is similar to that of Theorem~\ref{thm:main}. We need to divide the sum into two parts,
\begin{multline*}
\left \|(-\CD_{\a,\b,\g})^{r/2} g \right\|_{\a,\b,\g}^2 
=    \sum_{m=1}^\infty  |\l_{m}|^r \sum_{k= \lfloor \frac m 3 \rfloor+1}^m\Big|\wh g_{k,m}^{\a,\b,\g}\Big|^2 h_{k,m}^{\a,\b,\g} \\
     + \sum_{m=1}^\infty   |\l_{m}|^r \sum_{k= 0}^{\lfloor \frac m 3 \rfloor} \Big|\wh g_{k,m}^{\a,\b,\g}\Big|^2 h_{k,m}^{\a,\b,\g}.
\end{multline*}
For the first part, we use the estimate in \eqref{eq:i=3}, which holds for $m \ge r$, and it leads to 
\begin{align*}
    \sum_{m=1}^\infty  |\l_{m}|^r \sum_{k= \lfloor \frac m 3 \rfloor+1}^m\Big|\wh g_{k,m}^{\a,\b,\g}\Big|^2 h_{k,m}^{\a,\b,\g}\, & \lesssim \sum_{m=1}^{r-1} |\lambda_m|^r 
     \sum_{k= \lfloor \frac m 3 \rfloor+1}^m \left|\wh g_{k,m}^{\a,\b,\g}\right|^2 h_{k,m}^{\a,\b,\g} + 
          \left \| \phi_3^r \partial_3^r g \right \|_{\a,\b,\g} \\
    &  \lesssim   \|g\|_{\a,\b,\g} +  \left \| \phi_3^r \partial_3^r g \right \|_{\a,\b,\g}.
\end{align*}
For the second part, we also follow the proof of Theorem~\ref{thm:main} and notice that the estimate \eqref{eq:whf-whDf}
holds for $m \ge 3r -1$, so that a similar split as in the case of $i = 3$ appears, and we can conclude that 
\begin{align*}
 \sum_{m=1}^\infty   |\l_{m}|^r \sum_{k= 0}^{\lfloor \frac m 3 \rfloor} \Big|\wh g_{k,m}^{\a,\b,\g}\Big|^2 h_{k,m}^{\a,\b,\g}
    \lesssim   \|g\|_{\a,\b,\g} +  \left \| \phi_1^r \partial_1^r g \right \|_{\a,\b,\g} +  \left \| \phi_2^r \partial_2^r g \right \|_{\a,\b,\g}. 
\end{align*}
This completes the proof of \eqref{eq:der-bound2}. By the definition of $K$-functional, and by the use of the 
triangle inequality $\|g\|_{\a,\b,\g} \le \|f\|_{\a,\b,\g} + \|f- g\|_{\a,\b,\g}$, it is easy to see that \eqref{eq:der-bound2} 
implies \eqref{eq:K-equiv}.
\end{proof}

For $r =2$, the above theorem has been established in \cite{DHW} for $1 < p< \infty$ and, more generally, 
for the $d$-dimensional simplex.

\begin{rem}
Although the inverse estimate \eqref{eq:Bernstein} follows from the inverse estimate that holds for 
$K_r(f;t)_{L^2(\varpi_{\a,\b,\g})}$, because of \eqref{eq:K-equiv}, the direct estimate \eqref{eq:Jackson} cannot 
be deduced from the direct estimate for $K_r(f;t)_{L^2(\varpi_{\a,\b,\g})}$ because of the extra term 
$t^r \|f\|_{\a,\b,\g}$ in \eqref{eq:K-equiv}. Our proof indicates that the term $t^r \|f\|_{\a,\b,\g}$ in \eqref{eq:K-equiv} 
is necessary. 
\end{rem}
 
The standard proof of the inverse estimate \eqref{eq:Bernstein} (see, for example, \cite{DG}) shows that the 
estimate follows as a consequence of the Bernstein inequality. In our case, the proof follows from the 
inequalities in the following theorem. 
\begin{thm}
Let $r \in \NN$ and $\a,\b,\g > -1$. Then, for every $P_n \in \Pi_n^2$, we have
\begin{equation*}
   \left \| \phi_i^r \partial_i^r P_n \right \|_{\a,\b,\g} \le c\, n^{r} \|P_n\|_{\a,\b,\g}, \qquad 1 \le i \le 3,
\end{equation*}
where $c$ is a constant independent of $n$. 
\end{thm}

It is easy to see that these inequalities can be proved by following the proof of 
\eqref{eq:der-bound}.

\section{A family of Determinants}
\label{sec:4}
In this section we prove a closed form formula for a family of determinants, which includes the 
determinant that we need in Lemma~\ref{lem:2.2} in a special case. 

\begin{thm} \label{thm:1}
Define
$$
f(s_1, s_2, r, i, j) := 
 \binom r{ j - i}\frac { (s_1 + i)_{ j - i}}
{(s_1 + s_2 + i + j - 1)_{ j - i}\,
   (s_1 + s_2 + r + 2 i)_{ j - i}}
$$
and
$$
M(r_1,r_2):=\begin{pmatrix} 
 f(s_1, s_2, r_1, i, j)&\text{for }0\le i<r_2\\
 (-1)^{j - i - r_2} f(s_2, s_1, r_2, i - r_2, j)
&\text{for }r_2\le i<r_1+r_2-1
\end{pmatrix}_{0\le i,j\le r_1+r_2-1}.
$$
Then the determinant of $M(r_1,r_2)$ equals
\begin{equation} \label{eq:1}
(-1)^{r_1 r_2} 
\prod _{j=1} ^{r_1}
\frac {1} {   (s_1 + s_2 + r_1 + r_2 + j - 2)_{r_2}}.
\end{equation}
\end{thm}

\begin{proof}
The proof is based on the two lemmas given below.
If we write $M$ for $M(r_1,r_2)$ for short, the idea is to do a Laplace
expansion of $\det M$ with respect to the first $r_2$ rows,
\begin{equation} \label{eq:2}
\det M=
\sum _{0\le k_0<\dots <k_{r_2-1}\le r_1+r_2-1}
(-1)^{\binom {r_2}2+\sum_{i=0}^{r_2-1}k_i}
\det M^{k_0,\dots,k_{r_2-1}}_{0,\dots,r_2-1}
\cdot
\det M^{l_0,\dots,l_{r_1-1}}_{r_2,\dots,r_1+r_2-1},
\end{equation}
where $M^{a_1,\dots,a_r}_{b_1,\dots,b_r}$ denotes the submatrix
of $M$ consisting of rows $a_1,\dots,a_r$ and columns $b_1,\dots,b_r$,
and $\{l_0,\dots,l_{r_1-1}\}$ is the complement of
$\{k_0,\dots,k_{r_2-1}\}$ in $\{1,2,\dots, \break r_1+r_2-1\}$.
It turns out that both determinants on the right-hand side of \eqref{eq:2}
can be evaluated by means of Lemma~\ref{lem:1}. Thus, we obtain a multiple
sum for $\det M(r_1,r_2)$. This sum can then be evaluated by
Lemma~\ref{lem:2} and (considerable) simplification leads to the claimed
result on the right-hand side of \eqref{eq:1}.

So, if we use Lemma~\ref{lem:1} on the right-hand side of
\eqref{eq:2}, we obtain
\begin{multline} \label{eq:3}
\sum _{0\le k_0<\dots <k_{r_2-1}\le r_1+r_2-1}
(-1)^{\binom {r_2}2+\sum_{i=0}^{r_2-1}k_i}
\prod _{0\le i<j\le r_2-1} ^{}
(k_j - k_i) (k_i + k_j + s_1+s_2 - 1)
\\
\cdot
\prod _{i=0} ^{r_2-1}\left(
\frac {
(s_1)_{k_i}\,(s_1+s_2+r_1+2i-2)!\,(s_1+s_2+r_1+2i-1)!} 
{(s_1)_i\,(s_1+s_2 + r_1 + i - 2)!\,
(s_1+s_2+2k_i-2)!\,k_i!}\right.\\
\left.
\cdot
\frac {(r_1+i)!\,(s_1+s_2+k_i-2)!} 
{(r_1+r_2-k_i-1)!\,(s_1+s_2+r_1+r_2+k_i-2)!}
\right)\\
\cdot
(-1)^{\binom {r_1}2+\sum_{i=0}^{r_1-1}l_i}
\prod _{0\le i<j\le r_1-1} ^{}
(l_j - l_i) (l_i + l_j + s_1+s_2 - 1)
\kern2cm\\
\cdot
\prod _{i=0} ^{r_1-1}\left(
\frac {
(s_2)_{l_i}\,(s_1+s_2+r_2+2i-2)!\,(s_1+s_2+r_2+2i-1)!} 
{(s_2)_i\,(s_1+s_2 + r_2 + i - 2)!\,
(s_1+s_2+2l_i-2)!\,l_i!}\right.\\
\left.
\cdot
\frac {(r_2+i)!\,(s_1+s_2+l_i-2)!} 
{(r_1+r_2-l_i-1)!\,(s_1+s_2+r_1+r_2+l_i-2)!}
\right),
\end{multline}
where the $l_i$'s have the same meaning as before.
Clearly, we have
$$
\sum_{i=0}^{r_1-1}l_i=\binom {r_1+r_2}2-\sum_{i=0}^{r_2-1}k_i
$$
and
$$
\binom {r_1+r_2}2-\binom {r_1}2-\binom {r_2}2=r_1r_2.
$$
Furthermore, we use the inclusion/exclusion formulae
\begin{align*} 
\prod _{0\le i<j\le r_1-1} ^{}
(l_j - l_i) 
&=
\frac {\prod _{0\le i<j\le r_1+r_2-1} ^{}(j - i) 
\prod _{0\le i<j\le r_2-1} ^{} (k_j - k_i)
} 
{
\prod _{j=0} ^{r_2-1}
\prod _{i=0} ^{k_j-1}
(k_j-i)
\prod _{i=0} ^{r_2-1}
\prod _{j=k_i+1} ^{r_1+r_2-1}
(j-k_i)
}\\
&=\frac {
\prod _{i=0} ^{r_1+r_2-1}i!
\prod _{0\le i<j\le r_2-1} ^{} (k_j - k_i)
} 
{
\prod _{j=0} ^{r_2-1}k_j!
\prod _{i=0} ^{r_2-1}(r_1+r_2-k_i-1)!
}
\end{align*}
and
\begin{align*} 
\prod _{0\le i<j\le r_1-1} ^{}&
 (l_i + l_j + s_1+s_2 - 1)\\
&=
\frac {\prod _{0\le i<j\le r_1+r_2-1} ^{}(i+j+s_1+s_2-1) 
\prod _{0\le i<j\le r_2-1} ^{}
 (k_i + k_j + s_1+s_2 - 1)} 
{
\prod _{j=0} ^{r_2-1}
\prod _{i=0} ^{k_j-1}
(i+k_j+s_1+s_2-1)
\prod _{i=0} ^{r_2-1}
\prod _{j=k_i+1} ^{r_1+r_2-1}
(k_i+j+s_1+s_2-1)
}\\
&=
\frac {
\prod _{j=0} ^{r_1+r_2-1}(j+s_1+s_2-1)_{j}
\prod _{0\le i<j\le r_2-1} ^{}
 (k_i + k_j + s_1+s_2 - 1)
} 
{
\prod _{j=0} ^{r_2-1}
(k_j+s_1+s_2-1)_{k_j}
\prod _{i=0} ^{r_2-1}
(2k_i+s_1+s_2)_{r_1+r_2-k_i-1}
}.
\end{align*}
If we substitute all this in \eqref{eq:3}, then, upon further manipulation,
we arrive at
\begin{multline*} 
(-1)^{r_1r_2}
\prod _{i=0} ^{r_1+r_2-1}
\frac {(s_2)_i\,(s_1+s_2+i-2)!\,(i+s_1+s_2-1)_{i}} 
{(s_1+s_2+2i-2)!\,(r_1+r_2-i-1)!\,(s_1+s_2+r_1+r_2+i-2)!}\\
\times
\prod _{i=0} ^{r_2-1}
\frac {
(s_1+s_2+r_1+2i-2)!\,(s_1+s_2+r_1+2i-1)!\,(r_1+i)!} 
{(s_1)_i\,(s_1+s_2 + r_1 + i - 2)!\,(r_1+r_2-1)!\,(s_1+s_2)_{r_1+r_2-1}
}\\
\times
\prod _{i=0} ^{r_1-1}
\frac {
(s_1+s_2+r_2+2i-2)!\,(s_1+s_2+r_2+2i-1)!\,(r_2+i)!} 
{(s_2)_i\,(s_1+s_2 + r_2 + i - 2)!}\\
\times
\sum _{0\le k_0<\dots <k_{r_2-1}\le r_1+r_2-1}
(-1)^{\sum_{i=0}^{r_2-1}k_i}
\prod _{0\le i<j\le r_2-1} ^{}
(k_j - k_i)^2 (k_i + k_j + s_1+s_2 - 1)^2\\
\cdot
\prod _{i=0} ^{r_2-1}
\frac {(s_1+s_2-1+2k_i)} {(s_1+s_2-1)}\cdot
\frac {(s_1+s_2-1)_{k_i}\,(s_1)_{k_i}\,
(-r_1-r_2+1)_{k_i}
} 
{k_i!\,(s_2)_{k_i}\,
(s_1+s_2+r_1+r_2-1)_{k_i}}.
\end{multline*}
Now we apply Lemma~\ref{lem:2} with $r=r_2$, $a=s_1+s_2-1$, $b=s_1$,
and $m=r_1+r_2-1$. The result then finally condenses to the
right-hand side of~\eqref{eq:1}.
\end{proof}

\begin{lem} \label{lem:1}
For any positive integer $r_2$, we have
\begin{multline*} 
\det_{0\le i,j\le r_2-1}\left(
 \binom {r_1}{ k_j - i}\frac { (s_1 + i)_{ k_j - i}}
{(s_1 + s_2 + i + k_j - 1)_{ k_j - i}\,
   (s_1 + s_2 + r_1 + 2 i)_{ k_j - i}}
\right)\\
=
\prod _{0\le i<j\le r_2-1} ^{}
(k_j - k_i) (k_i + k_j + s_1+s_2 - 1)
\kern4cm\\
\times
\prod _{i=0} ^{r_2-1}\left(
\frac {
(s_1)_{k_i}\,(s_1+s_2+r_1+2i-2)!\,(s_1+s_2+r_1+2i-1)!} 
{(s_1)_i\,(s_1+s_2 + r_1 + i - 2)!\,
(s_1+s_2+2k_i-2)!\,k_i!}\right.\\
\left.
\cdot
\frac {(r_1+i)!\,(s_1+s_2+k_i-2)!} 
{(r_1+r_2-k_i-1)!\,(s_1+s_2+r_1+r_2+k_i-2)!}
\right).
\end{multline*}
\end{lem}

\begin{proof}
We take as many factors out of rows or columns such that inside
the determinant there remains a polynomial. More precisely, 
our determinant equals
\begin{multline*} 
\prod _{i=0} ^{r_2-1}
\frac {
(s_1)_{k_i}\,(s_1+s_2+r_1+2i-1)!\,r_1!\,(s_1+s_2+k_i-2)!} 
{(s_1)_i\,
(s_1+s_2+2k_i-2)!\,k_i!\,(r_1+r_2-k_i-1)!\,(s_1+s_2+r_1+r_2+k_i-2)!}\\
\times
\det_{0\le i,j\le r_2-1}\left(
(r_1 - k_j + i + 1)_{ r_2 - i - 1} \,
(s_1+s_2 + r_1 + i + k_j)_{ r_2 - i - 1}\right.\\
\left.
\cdot
(k_j - i + 1)_{ i}\, 
 ( s_1+s_2 + k_j - 1)_{ i}
\right).
\end{multline*}
We claim that the determinant in the last line equals
$$
\prod _{0\le i<j\le r_2-1} ^{}
(k_j - k_i) (k_i + k_j + S - 1)
\prod _{i=0} ^{r_2-1}
   (r_1 + 1)_{ i }\, (s_1+s_2 + r_1 + i - 1)_{ i }.
$$
This is seen by specializing $n=r_2$, $X_j=k_j$, $A_i=-(r_1+i-1)$, 
$C=-(s_1+s_2-1)$, and
$$
p_i(X)=(X-i+1)_i\,(s_1+s_2+X-1)_i
$$
in \cite[Lemma~7]{KratBN}.
\end{proof}

\begin{lem} \label{lem:2}
For all non-negative integers $m$, we have
\begin{multline*} 
\sum _{0\le k_1<\dots <k_r\le m} ^{}
(-1)^{\sum _{i=1} ^{r}k_i}
\prod _{1\le i<j\le r} ^{}(k_j-k_i)^2\,(k_i+k_j+a)^2\\
\cdot
\prod _{i=1} ^{r}\frac {(a+2k_i)} {a}\cdot
\frac {(a)_{k_i}\,(b)_{k_i}\,(-m)_{k_i}} 
{k_i!\,(a+1-b)_{k_i}\,(a+1+m)_{k_i}}\\
=
\prod _{i=1} ^{r}\frac {(i-1)!\,(b)_{i-1}\,m!\,(a+1)_m} 
{(m+1-i)!\,(a+1-b)_{m+1-i}}.
\end{multline*}
\end{lem}

\begin{proof}
This is a special case of a multi-dimensional ${}_{10}V_{9}$
summation formula conjectured by Warnaar (let $x=q$ in
\cite[Cor.~6.2]{WarnAG}), which has subsequently been
proven by Rosengren \cite{RosHAC} (and
in more generality by Rains~\cite[Theorem~4.9]{RainAA} and, 
independently,
by Coskun and Gustafson~\cite{CoGuAA}). Using the statement
of the identity in \cite[Theorem~3.1]{SchlAR}, we have to
first specialize $p=0$, then let $c,d\to\infty$, and finally
replace $a$ by $q^a$ and $b$ by $q^b$ and let $q\to1$.
\end{proof}

\end{document}